\documentclass{amsart}

\usepackage{amsmath, amssymb, amsfonts}
\usepackage{mathabx}

\newtheorem{theorem}{Theorem}[section]
\newtheorem{lemma}[theorem]{Lemma}
\newtheorem{proposition}[theorem]{Proposition}
\newtheorem{corollary}[theorem]{Corollary}

\theoremstyle{definition}
\newtheorem{definition}[theorem]{Definition}
\newtheorem{example}[theorem]{Example}

\theoremstyle{remark}
\newtheorem{remark}[theorem]{Remark}

\numberwithin{equation}{section}

\newcommand{\essup}{\operatorname{ess\;sup}}
\newcommand{\esinf}{\operatorname{ess\;inf}}

\begin{document}

\title[(p, q)-Frame measures on LCA groups]{(p, q)-Frame measures on LCA groups:
\\Perturbations and Construction}

\author[A. Tahmasebi Birgan]{$^{1}$Abdolreza Tahmasebi Birgani}
\address{$^1$Department of Mathematics, Faculty of Science, Central Tehran Branch,
Islamic Azad University, Tehran, Iran.}
\email{tahmasebi\_22@yahoo.com}
\thanks{This work was partially supported by the Central Tehran Branch, Islamic Azad University.}

\author[M. S. Asgari]{$^{2}$Mohammad Sadegh Asgari}
\address{$^2$Department of Mathematics, Faculty of Science, Central Tehran Branch,
Islamic Azad University, Tehran, Iran.}
\email{msasgari@yahoo.com ; moh.asgari@iauctb.ac.ir}


\date{Received: DDMMYY; Revised: DDMMYY; Accepted: DDMMYY.
\newline \indent $^{*}$Corresponding author}



\begin{abstract}
Motivated to generalize the Fourier frame concept to Banach spaces we introduce (p, q)-Bessel/frame measures
for a given finite measure on LCA groups. We also present a general way of constructing (p, q)-Bessel/frame
measures for a given measure. Moreover, we prove that if a measure has an associated (p, q)-frame measure,
then it must have a certain uniformity in the sense that the weight is distributed quite uniformly on its support.
Next, we show that if the measures $\mu$ and $\lambda$ without atoms whose supports form a packing pair,
then $\mu\ast\lambda+\delta_g\ast\mu$ does not admit any (p, q)-frame measure. Finally, we analyze the
stability of (p, q)-frame measures under small perturbations. We prove new theorems concerning the stability
of (p, q)-frame measures under perturbation in both Hilbert spaces and Banach spaces.
\end{abstract}

\maketitle

Let $\mu$ be a positive and finite Borel measure on $\mathbb{R}^d$ and let $e_\omega$ denote the exponential function
$e_\omega(x)=e^{2\pi i\langle\omega, x\rangle}$ $(\omega, x\in\mathbb{R}^d)$. We say that $\mu$ is a frame-spectral
measure if there exists a countable set $\Omega\subset\mathbb{R}^d$ such that a system of exponential functions $E(\Omega)
=\{e_\omega\}_{\omega\in\Omega}$ forms a Fourier frame for $L^2(\mu)$. If such $\Omega$ exists, then it is called a frame
spectrum of $\mu$. We say that $\mu$ is a spectral measure and $\Omega$ is a spectrum for $\mu$ if the system $E(\Omega)$
forms an orthonormal basis for $L^2(\mu)$. The existence of a frame-spectral measure $\mu$ for $\mathbb{R}^d$ with frame
spectrum $\Omega$ allows one to decompose in a stable way (but generally non-unique) any function $f$ from the space $L^2(\mu)$
in a Fourier series $f=\sum_{\omega\in\Omega}c_\omega e_{\omega}$ with frequencies belonging to $\Omega$ (see \cite{YOU}).\par
One of the basic questions in frame theory is to classify when a measure $\mu$ is spectral or frame-spectral. The origin of this question
goes back to Fuglede in 1974 \cite{FUG}, by investigating which subsets of $\mathbb{R}^d$ with the Lebesgue measures are spectrum
and to Jorgensen and Pedersen \cite{JOR} who discovered the existence of fractal spectral measures. It is well known, however, that the
existence of a spectral or frame-spectral measure is a strong requirement, satisfied by a relatively small class of measures. Hence it is of
interest to understand whether measures that are not spectral or frame-spectral. The notion of frame-spectral measures is a natural
generalization of an exponential orthonormal basis which Duffin and Schaeffer first introduced in \cite{DUF}. In \cite{DUT} the notion
of Fourier frames for an arbitrary measure was extended to that of a frame measure. Frame measures generalize the notion of
Fourier frame in the sense that if $\{e_\omega\}_{\omega\in\Omega}$ forms a Fourier frame for $L^2(\mu)$, then
$\nu=\sum_{\omega\in\Omega}\delta_\omega$ is a frame measure of $\mu$. The formulation of frame measures originates
from the notion of continuous frames in frame theory literature \cite{GAB}. It also was studied and used in \cite{DUTK, FU}.
Fourier frames nowadays have a wide range of applications in signal transmission and reconstruction. For the more general background
of Fourier frame theory, readers may refer to \cite{CHR, HE, LAB, LAI, LEV, NIT, MAR, STR}. \par
One of the most important concrete realizations of frames are  p-frames. p-Frames (Banach frames) were first systematically studied in
\cite{CHRI, FEI, GROC}. Aside from their theoretical appeal, frames in $L^p$ spaces and other Banach function spaces are effective
tools for modeling a variety of natural signals and images \cite{DEV, MAL}. They are also used in the numerical computation of
integral and differential equations.\par
While the theory of Fourier frames and Gabor analysis have usually been investigated on $\mathbb{R}^d$, there is an increasing
interest in the study of these theories in the context of general locally compact abelian (LCA) groups (for example see \cite{CAB, CHRS, GRO}).
When looking carefully at the theory of frame measures it becomes apparent that it is strongly based on the additive group operation
of $\mathbb{R}^d$. It is therefore interesting to see if the theory can be set in the context of general LCA groups. The LCA group framework has several advantages. First, because it is important to have a valid theory for classical groups such as $\mathbb{Z}^d$, $\mathbb{T}^d$ and $\mathbb{Z}_n$. Second, the LCA group setting unifies several different results into a general framework.
This will be crucial particularly in applications, as in the case of the generalization of the Fourier transform to LCA groups (see
\cite{CHRS, FOL, RUD}).\par
In this paper, we develop the theory of frame measures to the Banach spaces $L^p$ which we call the (p, q)-frame measure in
LCA groups. Our emphasis will be on construction and perturbation results. Indeed, we present a general way of constructing
(p, q)-Bessel/frame measures for a given measure. We prove that if a measure has an associated (p, q)-frame measure, it must
have a certain uniformity in the sense that the weight is distributed quite uniformly on its support. We also introduce a finite
Borel measure on LCA group $G$ which does not admit any (p, q)-frame measure. We also study the stability of (p, q)-frame
measures. Similar to ordinary frames, we show that (p, q)-frame measures are stable under small perturbations.\par
We first start by recalling the basic definitions and the necessary notations which will be used throughout the paper.\par
A LCA group is an arbitrary locally compact Hausdorff abelian group $G$ such that the group operation and inversion are
continuous. We will denote the group operation by $+$ and neutral element by $0$. For technical reasons we will assume
that $G$ is a countable union of compact sets and metrizable. The dual group of $G$, that is, the set of
continuous characters on $G$, is denoted by $\widehat{G}$. The value of the character $\gamma\in\widehat{G}$
at the point $x\in G$ is written by $\langle x, \gamma\rangle$. A character is a function $\gamma: G\rightarrow\mathbb{T}$,
where $\mathbb{T}$ is the complex numbers whose absolute value is $1$ such that for all $x, y\in G$: \begin{enumerate}
\item[$(i)$] $|\langle x, \gamma\rangle|=1$.\item[$(ii)$] $\langle x+y, \gamma\rangle=\langle x, \gamma\rangle\langle y,
\gamma\rangle$. \end{enumerate} As the name would suggest, endowed with the weak$^*$ topology and group operation
pointwise multiplication, $\widehat{G}$ is a LCA group. The character $\gamma$ of the group $G$ is called trivial if for every
$x \in G$, we have $\langle x, \gamma\rangle=1$. So every group $G$ have a trivial character. Moreover, every
character of $G$ mapping the neutral element of $G$ to 1, i.e. $\langle 0, \gamma\rangle=1$. We also have
\begin{align*} \overline{\langle x, \gamma\rangle}\langle x, \gamma\rangle=\vert\langle x, \gamma\rangle\vert^2 = 1.
\end{align*} Thus \begin{align*}\langle x, \overline{\gamma}\rangle
=\overline{\langle x, \gamma\rangle}=\langle x, \gamma\rangle^{-1}=\langle x, -\gamma\rangle=
\langle -x, \gamma\rangle. \end{align*} It follows that if $\gamma\in \widehat{G}$ then
$\overline{\gamma}\in \widehat{G}$. Moreover, if $G$ is a finite order group of $n$, then any character imaginig every
element of $G$ to the nth unit root. In this case we can traslated range of any character with set of all the nth unit root.
In the other words, \begin{align*} U_n = \lbrace \xi^{k}_{n} : \xi_n = e^{2\pi i/n}, 0 \leq k \leq n-1\rbrace. \end{align*}
$U_n$ is a cyclic subgroup of $\mathbb{C}^\ast$ with generator $\xi_n = e^{2\pi i/n}$, where $\mathbb{C}^\ast$ is
multiplication group of nonzero complex numbers. \par
The other fundamental feature is the existence and uniqueness of a Haar measure on an LCA group. A Haar measure on an
LCA group is a non-negative, regular Borel measure, which is not identically zero and is both left and right translation-invariant
and unique up to a scalar multiple. If $f$ is a Borel-measurable function on $G$, then we denote the integral of $f$ with
respect to Haar measure $\lambda_G$ by $\int_G f(x)d\lambda_G(x)$. As with other measure and topological spaces, we can
define the usual function spaces: $L^p(G)$ $(1< p<\infty)$, in the following way \begin{align*} L^p(G)=\Big\{
f : G\rightarrow\mathbb{C} :\; f\;\text{is measurable and} \int_G |f(x)|^pd\lambda_G(x)<\infty\Big\}.\end{align*}
If $G$ is a second countable LCA group, $L^p(G)$ is separable. \par
The rest of this paper is organized as follows: In Section 2 we extend the notion of frame measure to (p, q)-frame measure on
LCA groups and we give some examples of (p, q)-frame measures. We investigate some general results concerning (p, q)-frame
measures. In Section 3 we state some general way of constructing (p, q)-Bessel/frame measures for a given measure. Our emphasis
will be on the convolution of the Dirac measures and absolutely continuous measures with bounded Radon-Nikodym derivative.
We prove that if a measure has an associated (p, q)-frame measure, it must have a certain uniformity in the sense that the weight
is distributed quite uniformly on its support. We also introduce a finite Borel measure on LCA group $G$ which does not admit
any (p, q)-frame measure. In Section4, we analyze the stability of (p, q)-frame measures under small perturbations. we prove
new theorems concerning the stability of (p, q)-frame measures under perturbation in both Hilbert spaces and Banach spaces.

\section{(p, q)-Frame measures on LCA groups}

In this section, we extend the notion of frame measure to (p, q)-frame measure for any $(1< p, q<\infty)$ on LCA groups, and
we investigate some general results concerning (p, q)-frame measures.
\begin{definition}\label{def2.3}
Let $\mu, \nu$ be two Borel measures on LCA groups $G$ and $\widehat{G}$, respectively. Then \begin{enumerate}\item[$(i)$]
the Fourier transform of a function $f\in L^1(\mu)$ with respect to $\mu$ is defined by \begin{align}\label{eq2.1}
\widehat{fd\mu}:\widehat{G}\rightarrow\mathbb{C}, \hspace{1cm}\widehat{fd\mu}(\gamma)=\int_G f(x)
\overline{\langle x, \gamma\rangle}d\mu(x), \quad\forall \gamma\in\widehat{G}.\end{align}
By putting $f$ to be the characteristic function on the support of $\mu$, we obtain the Fourier-Stieltjes transform of a measure
$\mu$ \begin{align}\label{eq2.2}\widehat\mu(\gamma)=\int_G \overline{\langle x, \gamma\rangle}d\mu(x),\quad\forall
\gamma\in\widehat{G}.\end{align} \item[$(ii)$]  the inverse Fourier transform of a function $\varphi\in  L^1(\nu)$ with respect
to $\nu$ is given by \begin{align}\label{eq2.3} \widecheck{\varphi d\nu}: G\rightarrow\mathbb{C},\hspace{1cm}
\widecheck{\varphi d\nu} (x)=\int_{\widehat{G}} \varphi(\gamma)\langle x, \gamma\rangle d\nu(\gamma),\quad
\forall x\in G.\end{align}\end{enumerate}
\end{definition}
\begin{remark}\label{rem2.4}
Note that the Fourier-Stieltjes transform of a finite Borel measure  $\mu$ on LCA group $G$ is a bounded and uniformly
continuous function on $\widehat{G}$.
\end{remark}
Now, we extend the concepts of (p, q)-Bessel/frame measures as follows.
\begin{definition}\label{def2.5}
Let $1<p, q<\infty$. Let $\mu$ be a finite Borel measure on  LCA group $G$ and $\nu$ be a Borel measure on $\widehat{G}$,
respectively. Then the measure $\nu$ is called a (p, q)-frame measure for $\mu$ if there exist constants $0<A\leq B<\infty$
such that \begin{equation}\label{eq2.4} A\|f\|^q_{L^p(\mu)}\leq\|\widehat{fd\mu}\|^q_{L^q(\nu)}\leq B\|f\|^q_{L^p(\mu)},
\hspace{0.5cm}\forall f\in L^p(\mu). \end{equation} $A$ and $B$ are called frame bounds for $\nu$. If $p=q=2$, $\nu$ is
called a frame measure for $\mu$. The measure $\nu$ is called a (p, q)-Bessel measure for $\mu$, when it satisfies at least
the upper (p, q)-frame measure inequality. If $A=B=1$, $\nu$ is called a (p, q)-Plancherel measure for $\mu$.
\end{definition}
\begin{example}\label{exam2.8}
Let $x_k\in G$, $k=1,\cdots, n$, and let $\Omega\subset\widehat{G}$ be a countable subset. Suppose that
$\{c_\omega\}_{\omega\in\Omega}$ is a sequence of positive real numbers such that $0<A\leq\sum_{\omega\in\Omega}
c_\omega\leq B<\infty$. Then the measure $\nu=\sum_{\omega\in\Omega}c_\omega\delta_\omega$ is a (p, q)-frame
measure for $\mu=\delta_{x_1}\ast\cdots\ast\delta_{x_n}$ with frame bounds $A$, $B$, where $\delta_x$ is the Dirac
measure at $x$.\end{example}
The following example illustrates the existence of a (p, q)-Bessel measure for an infinite continuous measure on an LCA
group $G$.
\begin{example}\label{exam2.81}
Let $\lambda_G$, $\lambda_{\widehat{G}}$ be the Haar measures on LCA groups $G$ and $\widehat{G}$, respectively.
Suppose $1<p \leq 2$ and $p^{-1}+q^{-1}=1$. Then by the Hausdorff-Young Inequality \cite{10}, $\lambda_{\widehat{G}}$
is a (p, q)-Bessel measure for $\lambda_G$.
\end{example}
\begin{example}\label{exam2.9}
Let $\lambda_G$, $\lambda_{\widehat{G}}$ be the Haar measures on LCA groups $G$ and $\widehat{G}$, respectively.
Suppose that $E\subset G$ is a set with positive and finite Haar measure, and also assume that $\varphi\in L^1(\widehat{G})$
satisfies that $0<A\leq\varphi(\gamma)\leq B<\infty$ a.e., on $\widehat{G}$. Define $d\mu=\chi_Ed\lambda_G$ and $d\nu
=\varphi d\lambda_{\widehat{G}}$. Then $\nu$ is a frame measure for $\mu$ with frame bounds $A$, $B$.
 Indeed, for every $f\in L^2(\mu)$ we have \begin{align*}\int_G|f(x)|^2d\mu(x)=\int_G\chi_E(x)|f(x)|^2d\lambda_G(x)
=\int_G|(\chi_Ef)(x)|^2d\lambda_G(x). \end{align*} This gives that $\chi_Ef\in  L^2(G)$ and $\|f\|_{L^2(\mu)}=
\|\chi_Ef\|_{L^2(G)}$. On the other hand, for every $\gamma\in\widehat{G}$ we obtain \begin{align*} \widehat{fd\mu}
(\gamma) &=\int_Gf(x)\overline{\langle x, \gamma\rangle}d\mu(x)=\int_G\chi_E(x)f(x)\overline{\langle x, \gamma
\rangle}d\lambda_G(x)=\widehat{(\chi_Ef)}(\gamma). \end{align*} Now by the Plancherel theorem we observe that
\begin{align*} A\|f\|^2_{L^2(\mu)}&=A\|\chi_Ef\|^2_{L^2(G)}=A\|\widehat{(\chi_Ef)}\|^2_{L^2(\widehat{G})}\\&
\leq\int_{\widehat{G}}\varphi(\gamma)\big|\widehat{(\chi_Ef)}(\gamma)\big|^2d\lambda_{\widehat{G}}(\gamma)
=\|\widehat{fd\mu}\|_{L^2(\nu)}^2 \\&\leq B\|\widehat{(\chi_Ef)}\|^2_{L^2(\widehat{G})}= B\|\chi_Ef\|^2_{L^2(G)}
=B\|f\|^2_{L^2(\mu)}. \end{align*}\end{example}
\begin{remark}\label{rem2.10} Let $1<p, q<\infty$. Let $\mu$ be a finite Borel measure on LCA group $G$.
Then the family of (p, q)-Bessel measures for $\mu$ is never empty. Let $\Omega\subseteq\widehat{G}$ be a countable
subset and $\{c_\omega\}_{\omega\in\Omega}$ is a sequence of positive real numbers such that $\sum_{\omega\in\Omega} c_\omega\leq B/\mu(G)$. Then the measure $\nu=\sum_{\omega\in\Omega} c_\omega\delta_\omega$ is a (p, q)-Bessel
measure for $\mu$ with Bessel bound $B$. Indeed, for every $f\in L^p (\mu)$,  we have \begin{align*} \int_{\widehat{G}}
|\widehat{fd\mu} (\gamma)|^q d\nu(\gamma)&=\sum_{\omega\in\Omega} c_\omega\int_{\widehat{G}}|\widehat{fd\mu}
(\gamma)|^q d\delta _\omega (\gamma)=\sum_{\omega\in\Omega} c_\omega |\widehat{fd\mu}(\omega)|^q \\&
\leq\sum_{\omega\in\Omega} c_\omega\|f\|_{L^p(\mu)}^q\mu(G)\leq B \|f\|_{L^p(\mu)}^q. \end{align*} \end{remark}
\begin{definition}\label{def2.11}
Let $1<p, q<\infty$ and let $\mu$ be a finite Borel measure on LCA group $G$. Then $\mu$ is called a (p, q)-frame-spectral
measure if there exists a countable set $\Omega\subset\widehat{G}$ such that $\Omega$ is a q-frame for $L^p(\mu)$, i.e.,
there exist two constants $0<A\leq B<\infty$ such that \begin{equation}\label{eq2.5}
A\|f\|_{L^p(\mu)}^q \leq\sum_{\omega\in\Omega}|\widehat{fd\mu}(\omega)|^q\leq B\|f\|_{L^p(\mu)}^q,
\hspace{1cm} \forall f\in L^p(\mu). \end{equation} When $\mu$ is a (p, q)-frame-spectral measure, then $\Omega$ is called
a (p, q)-frame-spectrum for $\mu$.
\end{definition}
\begin{remark}\label{rem2.12}
In general, if $\mu$ is a (p, q)-frame-spectral measure with (p, q)-frame-spectrum $\Omega$, then $\nu=\sum_{\omega\in\Omega}\delta_\omega$
is a (p, q)-frame measure for $\mu$. Indeed, if $A, B$ are the frame
bounds for q-frame $\Omega$. Then for every $f\in L^p(\mu)$ we have \begin{align*} A\| f\|_{L^p(\mu)}^q&\leq
\sum_{\omega\in\Omega} |\widehat{fd\mu}(\omega)|^q=\int_{\widehat{G}}|\widehat{fd\mu}(\gamma)|^q
d\nu(\gamma)\leq B\|f\|^q_{L^p(\mu)}. \end{align*} Conversely, if $\nu$ is purely atomic and is a (p, q)-frame
measure for finite Borel measure $\mu$, i.e. there exists a countable subset $\Omega\subseteq\widehat{G}$ such that
$\nu=\sum_{\omega\in\Omega} c_\omega\delta_\omega$ and $c_\omega>0$, then the set $\Lambda=
\big\{c_\omega^{\frac{1}{q}}\omega : \;\omega\in\Omega\big\}$ is a q-frame for $L^p(\mu)$. Indeed, for each $f\in L^p(\mu)$
we have \begin{align*}\sum_{\gamma\in\Lambda}|\widehat{fd\mu}(\gamma)|^q=\sum_{\omega\in\Omega}
|\widehat{fd\mu}(c_\omega^{\frac{1}{q}}\omega)|^q&=\int_{\widehat{G}}|\widehat{fd\mu}(\gamma)|^q
d\nu(\gamma)=\|\widehat{fd\mu}\|^q_{L^q(\nu)}. \end{align*} This shows that $\mu$ is a (p, q)-frame-spectral
measure with (p, q)-frame-spectrum $\Lambda$. \end{remark}

\section{Characterization and Construction of (p, q)-frame measures}

In this section we present a general way of constructing (p, q)-Bessel/frame measures for a given measure. We prove that
if a measure has an associated (p, q)-frame measure, it must have a certain uniformity in the sense that the weight is distributed
quite uniformly on its support. We also introduce a finite Borel measure on LCA group $G$ which does not admit any (p, q)-frame
measure. First, we show that every (p, q)-Bessel measure must be locally finite. So, if $\widehat{G}$ is connected, then any
(p, q)-Bessel measure is $\sigma$-finite. Indeed, we can generalize Proposition 2.1 of \cite{6} to the situation of (p, q)-Bessel
measures.
\begin{proposition}\label{pro3.1}
Let $\mu$ be a finite Borel measure on LCA group $G$ and $\nu$ be a Borel measure on $\widehat{G}$, respectively.
Let $\nu$ be a (p, q)-Bessel measure for $\mu$. Then for any compact subset $K\subset\widehat{G}$, there exists a
constant $C_K$ such that $\nu(K) \leq C_K$. Consequently, if $\widehat{G}$ is connected, then $\nu$ is $\sigma$-finite. \end{proposition}\begin{proof} Let $B$ be the Bessel bound for the (p, q)-Bessel measure $\nu$.
Since $\widehat{\mu}(0)=\mu(G)> 0$, by the uniformly continuous of $\widehat{\mu}$ there is a neighborhood $V$
of point zero in $\widehat{G}$ and $\delta > 0$ with $|\widehat{\mu}(\gamma)|\geq \delta$, for all $\gamma\in V$.
Moreover, for every $\xi\in\widehat{G}$ we have \begin{align*}\widehat{\xi d\mu}(\gamma) &=\int_G \langle y, \xi\rangle
\overline{\langle y, \gamma\rangle}d\mu(y)= \int _G \overline{\langle y, \gamma- \xi\rangle}d\mu(y)
=\widehat{\mu}(\gamma- \xi).\end{align*} This implies that \begin{align*} B\mu(G)^{\frac{q}{p}} &=
B\|\xi\|_{L^p (\mu)}^q \geq\|\widehat{\xi d\mu}\|_{L^q (\nu)}^q =\int_{\widehat{G}}
|\widehat{\xi d\mu}(\gamma)|^q d\nu(\gamma)\\&=\int_{\widehat{G}}|\widehat{\mu}(\gamma- \xi)|^q d \nu(\gamma)
\geq\int_{\xi+V}|\widehat{\mu}(\gamma- \xi)|^q d\nu(\gamma)\geq\delta^q\nu (\xi+V). \end{align*}
Therefore $\nu(\xi+V)\leq {B\mu(G)^{\frac{q}{p}} / \delta^q} $ for all $\xi\in\widehat{G}$. Now suppose that
$K\subseteq\widehat{G}$ is compact. Then there exist indices $1\leq i \leq n$ such that $K\subseteq
\bigcup_{i=1}^{N_K}(\xi_i+V)$. So we have \begin{align*} \nu(K) \leq \sum_{i=1}^{N_K}\nu(\xi_i+V) \leq
\sum _{i=1}^{N_K} \frac{B\mu(G)^{\frac{q}{p}}}{\delta^q}=\frac{B\mu(G)^{\frac{q}{p}}}{\delta^q}
N_K=C_K. \end{align*} Now, if $\widehat{G}$ is connected, then $\nu$ is a $\sigma$-compact and so $\nu$ is $\sigma$-finite.
\end{proof}
The next proposition shows that each finite Borel measure on $\widehat{G}$ is a (p, q)-Bessel measure for any finite Borel
measure on $G$.
\begin{proposition}\label{pro3.2}
Let $\mu$ and $\nu$ be two finite Borel measures on LCA groups $G$ and $\widehat{G}$, respectively. Then $\nu$ is a
(p, q)-Bessel measure for $\mu$.
\end{proposition}
\begin{proof}
Let $p'$ be the conjugate exponent to $p$. Then for every $f \in L^p(\mu)$ and $\gamma\in\widehat{G}$ we have
\begin{align*} \big|\widehat{fd\mu}(\gamma)\big|^q &\leq\|f\|_{L^p(\mu)}^q\|\gamma\|_{L^{p'}(\mu)}^q=
(\mu(G))^{\frac{q}{p'}}\|f\|_{L^p(\mu)}^q.\end{align*} Thus \begin{align*} \big\|\widehat{fd\mu}\big\|^q_{L^q(\nu)}&
=\int_{\widehat{G}}\big|\widehat{fd\mu}(\gamma)\big|^q d\nu(\gamma)\leq\int_{\widehat{G}}(\mu(G))^{\frac{q}{p'}}
\|f\|_{L^p(\mu)}^q d\nu(\gamma)\\&=\nu(\widehat{G})(\mu(G))^{\frac{q}{p'}}\|f\|^q_{L^p(\mu)}.\end{align*}
\end{proof}
In the following results, we show that the (p, q)-frame measure property is preserved under the convolution of the Dirac
measures. First, we show that the (p, q)-frame measure property is preserved under the convolution of the Dirac measures
on $G$.
\begin{theorem}\label{thm3.3}
Let $\mu$ be a finite Borel measure on $G$ and $x\in G$. Then a Borel measure $\nu$ on $\widehat{G}$ is a (p, q)-frame
measure for $\mu$ if and only if $\nu$ is a (p, q)-frame measure for $\delta_x\ast\mu$ with the same frame bounds.
\end{theorem}
\begin{proof}
Let $\nu$ be a (p, q)-frame measure for $\delta_x\ast\mu$ with frame bounds $A$ and $B$. Then for every $f\in L^p(\mu)$
we have \begin{align*} \int_G|f(y-x)|^pd(\delta_x\ast\mu)(y)=\int_G\int_G|f(y+z-x)|^pd\delta_x(y)d\mu(z)=
\int_G|f(z)|^pd\mu(z).\end{align*} This gives that $f(\cdot-x)\in L^p(\delta_x\ast\mu)$ and $\|f\|_{L^p(\mu)}
=\|f(\cdot-x)\|_{L^p(\delta_x\ast\mu)}$. On the other hand, \begin{align*} \big|\widehat{f(\cdot-x)d(\delta_x\ast\mu)}
(\gamma)\big|&=\Big|\int_Gf(y-x)\overline{\langle y, \gamma\rangle}d(\delta_x\ast\mu)(y)\Big|\\&=
\Big|\int_Gf(y-x)\overline{\langle y-x, \gamma\rangle}d(\delta_x\ast\mu)(y)\Big|\\&=
\Big|\int_G\int_Gf(y+z-x)\overline{\langle y+z-x, \gamma\rangle}d\delta_x(y)d\mu(z)\Big|\\&=
\Big|\int_Gf(z)\overline{\langle z, \gamma\rangle}d\mu(z)\Big|=\big|\widehat{fd\mu}(\gamma)\big|. \end{align*}
Putting $f(\cdot-x)$ into the definition of (p, q)-frame measure in the hypothesis, we obtain that $\nu$ is a (p, q)-frame measure for
$\mu$. The converse is similar. \end{proof}
A similar result shows that the (p, q)-frame measure property is preserved under the convolution of the Dirac measures on
$\widehat{G}$.
\begin{theorem}\label{thm3.4}
Let $\mu$ be a finite Borel measure on $G$ and $\omega\in\widehat{G}$. Then a Borel measure $\nu$ on $\widehat{G}$ is
a (p, q)-frame measure for $\mu$ if and only if $\delta_\omega\ast\nu$ is a (p, q)-frame measure for $\mu$ with the same
frame bounds. \end{theorem}
\begin{proof}
Let $\delta_\omega\ast\nu$ be a (p, q)-frame measure for $\mu$ with frame bounds $A$ and $B$. Then for every $f\in L^p(\mu)$
and $\omega, \gamma\in\widehat{G}$ we have \begin{align*} \widehat{(\omega f)d\mu}(\omega+\gamma)&=\int_G
(\omega f)(x)\overline{\langle x, \omega+\gamma\rangle}d\mu(x)\\&=\int_G |\langle x, \omega\rangle|^2 f(x)
\overline{\langle x, \gamma\rangle}d\mu(x)\\&=\widehat{fd\mu}(\gamma).\end{align*}
From this we obtain  \begin{align*}\int_{\widehat{G}}\big|\widehat{(\omega f)d\mu}(\gamma)\big|^q
d(\delta_\omega\ast\nu)(\gamma)&=\int_{\widehat{G}}\int_{\widehat{G}}\big|\widehat{(\omega f)d\mu}(\zeta+\gamma)
\big|^qd\delta_\omega(\zeta)d\nu(\gamma)\\&=\int_{\widehat{G}}\big|\widehat{(\omega f)d\mu}(\omega+\gamma)
\big|^qd\nu(\gamma)\\&=\int_{\widehat{G}}\big|\widehat{fd\mu}(\gamma)\big|^qd\nu(\gamma). \end{align*}
Since $\|\omega f\|_{L^p(\mu)}=\|f\|_{L^p(\mu)}$, this implies that $\nu$ is a (p, q)-frame measure for $\mu$ with
the same frame bounds. The converse is similar. \end{proof}
As a consequence of Theorems \ref{thm3.3} and \ref{thm3.4}, we show that the (p, q)-frame measure property is
preserved under the convolution of the Dirac measures. We leave the proof to interested readers.
\begin{corollary}\label{cor3.5}
Let $x\in G,\;\omega\in\widehat{G}$. Then $\nu$ is a (p, q)-frame measure for $\mu$ if and only if $\delta_\omega\ast\nu$ is
a (p, q)-frame measure for $\delta_x\ast\mu$ with the same frame bounds.
\end{corollary}
\begin{theorem}\label{thm3.6}
Let $\nu$ be a (p, q)-frame measure for $\mu$. Then $\nu\ast\rho$ is a (p, q)-frame measure for $\mu$ for all finite Borel
measure $\rho$ on $\widehat{G}$.\end{theorem}
\begin{proof}
Let $\nu$ be a (p, q)-frame measure for $\mu$ with frame bounds $A$ and $B$. Then for every $f\in L^p(\mu)$ and
$\omega, \gamma\in\widehat{G}$ we have \begin{align*}
\widehat{fd\mu}(\omega+\gamma)&=\int_G f(x) \overline{\langle x, \omega+\gamma\rangle}d\mu(x)\\&
=\int_G (-\omega f)(x)\overline{\langle x, \gamma\rangle}d\mu(x)=\widehat{(-\omega f)d\mu}(\gamma).
\end{align*} This gives \begin{align*}
\big\|\widehat{fd\mu}\big\|^q_{L^q(\nu\ast\rho)}&=\int_{\widehat{G}}\big|\widehat{fd\mu}(\gamma)\big|^q
d(\nu\ast\rho)(\gamma)\\&=\int_{\widehat{G}}\int_{\widehat{G}}\big|\widehat{fd\mu}(\gamma+\omega)\big|^q
d\nu(\gamma)d\rho(\omega)\\&=\int_{\widehat{G}}\int_{\widehat{G}}\big|\widehat{(-\omega f)d\mu}(\gamma)\big|^q
d\nu(\gamma)d\rho(\omega)\\&=\int_{\widehat{G}}\|\widehat{(-\omega f)d\mu}\|^q_{L^q(\nu)}
d\rho(\omega).\end{align*} Since $\|(-\omega f)\|_{L^p(\mu)}=\|f\|_{L^p(\mu)}$, this implies that $\nu\ast\rho$ is a
(p, q)-frame measure for $\mu$ with frame bounds $A\rho(\widehat{G})$ and $B\rho(\widehat{G})$.
\end{proof}
\begin{definition}\label{def3.07}
Let $\mu, \lambda$ be two finite Borel measures on LCA group $G$. Then \begin{enumerate}\item[$(i)$] the closed support of
$\mu$ is the smallest closed set $K\subset G$ such that $\mu$ has full measure (i.e. $\mu(K)=\mu(G)$). We will denote it by
$K_\mu$. Also if  $\rho=\mu\ast\lambda$, then $K_\rho=K_\mu+K_\lambda$. \item[$(ii)$] A Borel set $X\subset G$ is called
a Borel support of $\mu$ if $X\subset K_\mu$ and $\mu(X)=\mu(K_\mu)$. \item[$(iii)$] We say that the pair $(\mu, \lambda)$
is a packing pair of measures if \begin{align*} (K_\mu - K_\mu)\cap (K_\lambda - K_\lambda)=\{0\}, \end{align*} where
$K_\mu+K_\lambda$ is the standard Minkowski sum of two sets.
\end{enumerate}\end{definition}
The following results show that the (p, q)-frame measure property is preserved with respect to absolutely continuous measures
with bounded Radon-Nikodym derivative.
\begin{theorem}\label{thm3.7}
Let $\mu$ be a finite Borel measure on $G$ and let $\phi$ be a $\mu$-measurable function such that $0<A\leq\phi(x)\leq
B<\infty$ $\mu$-a.e. on $K_\mu$. Then a Borel measure $\nu$ on $\widehat{G}$ is a (p, q)-frame measure for $\mu$ if
and only if $\nu$ is a (p, q)-frame measure for $\phi d\mu$.
\end{theorem}\begin{proof}
Let $\nu$ be a (p, q)-frame measure for $\mu$ with frame bounds $C$ and $D$. Then for any $f\in L^p(\phi d\mu)$
we have \begin{align*} A^{q(1-\frac{1}{p})}\|f\|^q_{L^p(\phi d\mu)}\leq\|\phi f\|^q_{L^p (\mu)}\leq B^{q(1-\frac{1}{p})}
\|f\|^q_{L^p (\phi d\mu)}.\end{align*} This gives that $f\in L^p(\phi d\mu)$ if and only if $\phi f\in L^p(\mu)$. On the
other hand, we obtain
\begin{align*} \widehat{fd(\phi d\mu)}(\gamma)=\int_Gf(x)\overline{\langle x, \gamma\rangle}\phi(x)d\mu(x)=
\widehat{(\phi f)d\mu}(\gamma). \end{align*} This yields \begin{align*} CA^{q(1-\frac{1}{p})}\|f\|^q_{L^p(\phi d\mu)}
&\leq C\|\phi f\|^q_{L^p(\mu)}\leq\|\widehat{(\phi f)d\mu}\|^q_{L^q(\nu)}\\&=\|\widehat{fd(\phi d\mu)}\|^q_{L^q(\nu)}
\leq D\|\phi f\|^q_{L^p(\mu)}\leq DB^{q(1-\frac{1}{p})}\|f\|^q_{L^p(\phi d\mu)}. \end{align*} The opposite direction
is obvious. Indeed, if $\nu$ is a (p, q)-frame measure for $\phi d\mu$ with frame bounds $C$ and $D$, the calculation
above implies that $\nu$ is a (p, q)-frame measure for $\mu$ with frame bounds $CB^{-q(1-\frac{1}{p})}$ and
$DA^{-q(1-\frac{1}{p})}$
\end{proof}
The following results are a slight variation of Theorem \ref{thm3.7}. We leave the proofs to interested readers.
\begin{theorem}\label{thm3.8}
Let $\mu$ be a finite Borel measure on $G$ and $\nu$ be a Borel measure on $\widehat{G}$. Let $\psi$ is a $\nu$-measurable
function such that $0<A\leq\psi(\gamma)\leq B<\infty$ $\nu$-a.e. on $K_\nu$. Then $\nu$ is a (p, q)-frame measure for $\mu$
if and only if $\psi d\nu$ is a (p, q)-frame measure for $\mu$.
\end{theorem}
\begin{corollary}\label{cor3.9}
Let $\mu$ be a finite Borel measure on $G$ and $\nu$ be a Borel measure on $\widehat{G}$. Let $\phi, \psi$ be two measurable
functions on $G$ and $\widehat{G}$ such that \begin{align*} 0<A\leq\phi(x)\leq B<\infty,\hspace{0.5cm}\text{and}
\hspace{0.5cm}0<C\leq\psi(\gamma)\leq D<\infty, \end{align*} a.e. on $K_\mu$ and $K_\nu$, respectively.
Then $\nu$ is a (p, q)-frame measure for $\mu$ if and only if $\psi d\nu$ is a (p, q)-frame measure for $\phi d\mu$.
\end{corollary}
The following theorem provides the connection between the existence of (p, q)-frame measure between the measures $\mu,
\lambda$, and the sum $\mu +\lambda$. This theorem is a generalization of [12, Lemma 2.1] to the situation of (p, q)-frame
measures. We will use it later.
\begin{theorem}\label{thm3.10}
Let $\mu, \lambda$ be two finite Borel measures on $G$ and $\nu$ be a Borel measure on $\widehat{G}$. Suppose
that $\mu(K_\lambda)=0$. If $\nu$ is a (p, q)-frame measure for $\mu+\lambda$, then $\nu$ is a (p, q)-frame measure for
$\mu$ and $\lambda$ with the same frame bounds. \end{theorem}
\begin{proof}
Let $\nu$ be a (p, q)-frame measure for $\mu+\lambda$ with frame bounds $A$ and $B$ and $\mu(K_\lambda)=0$.
For any $f\in L^p(\lambda)$, we may take $f(x)=0$ for all $x$ outside $K_\lambda$. Then \begin{align*} \int_G|f(x)|^p
d(\mu+\lambda)(x)&=\int_{K_\lambda}|f(x)|^pd\mu(x)+\int_G|f(x)|^pd\lambda(x)\\&=\int_G|f(x)|^pd\lambda(x)
\end{align*} This gives that $f\in L^p(\mu+\lambda)$ and $\|f\|_{L^p(\mu+\lambda)}=\|f\|_{L^p(\lambda)}$.
On the other hand, for every $\gamma\in\widehat{G}$ we have \begin{align*} \widehat{fd(\mu+\lambda)}(\gamma)
&=\int_G f(x)\overline{\langle x, \gamma\rangle}d(\mu+\lambda)(x)\\&=\int_{K_\lambda} f(x)
\overline{\langle x,\gamma\rangle}d\mu(x)+\int_G f(x)\overline{\langle x, \gamma\rangle}
d\lambda(x)\\&=\int_G f(x)\overline{\langle x, \gamma\rangle}d\lambda(x)=\widehat{fd\lambda}(\gamma). \end{align*}
This yields \begin{align*} A\|f\|^q_{L^p(\lambda)}&=A\|f\|^q_{L^p(\mu+\lambda)}\leq\int_{\widehat{G}}
\big|\widehat{fd(\mu+\lambda)}(\gamma)\big|^qd\nu(\gamma)\\&=\int_{\widehat{G}}\big|\widehat{fd\lambda}
(\gamma)\big|^qd\nu(\gamma)\leq B\|f\|^q_{L^p(\mu+\lambda)}=B\|f\|^q_{L^p(\lambda)}. \end{align*}
Furthermore, for any $f\in L^p(\mu)$, we define \begin{align*} g(x)=\left\{\begin{array}{ll} 0& x\in K_\lambda
\\ f(x) & x\in K_\mu\setminus K_\lambda. \end{array}\right. \end{align*} Then we have \begin{align*} \int_G|g(x)|^p
d(\mu+\lambda)(x)&=\int_{K_\mu\setminus K_\lambda}|f(x)|^pd\mu(x)+\int_{K_\lambda}|g(x)|^pd\lambda(x)
\\&=\int_{K_\mu\setminus K_\lambda}|f(x)|^pd\mu(x)\\&=\int_{K_\mu\setminus K_\lambda}|f(x)|^pd\mu(x)
+\int_{K_\lambda}|f(x)|^pd\mu(x)\\&=\int_G|f(x)|^pd\mu(x). \end{align*} This implies that $g\in L^p(\mu+\lambda)$
and $\|g\|_{L^p(\mu+\lambda)}=\|f\|_{L^p(\mu)}$. Moreover, for every $\gamma\in\widehat{G}$ we also have
$\widehat{gd(\mu+\lambda)}(\gamma)=\widehat{fd\mu}(\gamma)$. Therefore $\nu$ is a (p, q)-frame measure for
$\mu$ with frame bounds $A, B$. \end{proof}
In \cite{7}, Dutkay and Lai by considering certain absolute continuity properties of the measure and its translation recovered the characterization on absolutely continuous measures with Fourier frames. They proved that if a measure has an associated frame
measure, then it must have a certain uniformity in the sense that the weight is distributed quite uniformly on its support. In the
following, we recover this characterization for the (p, q)-frame measure situation. We will use the standard notation $\mu\ll\lambda$
to indicate that $\mu$ is absolutely continuous with respect to $\lambda$ and we use the notation $\frac{d\mu}{d\lambda}$ for
its Radon-Nikodym derivative if $\lambda$ is $\sigma$-finite.
\begin{definition}\label{def3.011}
Let $\mu$ be a Borel measure on LCA group $G$ and let $a\in G$ and $E\subset G$ be a Borel subset. We denote by $\mu|_E$
the restriction of $\mu$ to the set $E$ and $T_a\mu$ the translation by $a$ of $\mu$, i.e., for every Borel subset $F\subset G$
\begin{align*} \mu|_{E}(F)=\mu(E\cap F)\hspace{0.5cm}\text{and}\hspace{0.5cm}T_a\mu(F)=\mu(F+a).\end{align*} This
means that \begin{align*} \int_G f(x)dT_a(\mu|_E)(x)=\int_E f(x-a)d\mu(x),\end{align*} for all functions $f\in
L^1\big(T_a(\mu|_E)\big)$.
\end{definition}
\begin{theorem}\label{thm3.11}
Let $\mu$ be a finite Borel measure on LCA group $G$ and $\nu$ be a Borel measure on $\widehat{G}$, respectively. Let $\nu$
be a (p, q)-frame measure for $\mu$ with (p, q)-frame bounds $A, B>0$. Assume in addition that there exist a set $F\subset G$
with $\mu(F)>0$ and $a\in G$ such that $T_a(\mu|_{F+a})\ll\mu$. Then $\big\|\frac{dT_a(\mu|_{F+a})}{d\mu}
\big\|^{q(1-\frac{1}{p})}_\infty\leq\frac{B}{A}$.\end{theorem}\begin{proof}
Let $\phi =\frac{dT_a(\mu|_{F+a})}{d\mu}$ and let $M=\|\phi\|_{\infty}$. First, note that restricting $\mu$
to a subset of $F$, we can assume $1\leq M< \infty$. Let $f:G\rightarrow\mathbb{C}$ be a bounded function. Then we have
\begin{align*} \int_G f(x)\phi(x)d\mu(x)&=\int_G f(x)dT_a(\mu|_{F+a})(x)=\int_G f(x-a) d\mu|_{F+a}(x)\\&=
\int_{F+a} f(x-a) d\mu(x). \end{align*} Therefore the values of the functions $f$ and $\phi$ can be ignored outside $F$ and
so we can assume $f$ and $\phi$ are supported on $F$. Fix $\epsilon>0$. Since $M$ is the essential supremum of $\phi$,
there exists a subset $E\subset F$ with $\mu(E)>0$ such that $M-\epsilon\leq\phi\leq M$ on $E$. Now take $g:G\rightarrow
\mathbb{R}$ by $g=\frac{1}{\sqrt[p]{\mu(E)}}\chi_E$. We have $\|g\|_{L^p(\mu)}=1$ and \begin{align*}\|g(\cdot-a)\|^p
_{L^p(\mu)}&=\int_G |g(x-a)|^pd\mu(x)=\int_{F+a}|g(x-a)|^pd\mu(x)\\&=\int_G |g(x)|^p\phi(x)d\mu(x). \end{align*}
Thus $M-\epsilon\leq\|g(\cdot-a)\|^p_{L^p(\mu)}\leq M$. On the other hand, we have \begin{align*}\big|
\widehat{g(\cdot-a) d\mu}(\gamma)\big|&=\Big|\int_G g(x-a)\overline{\langle x, \gamma\rangle} d\mu(x)\Big|
=\Big|\int_G g(x-a)\overline{\langle x-a, \gamma\rangle}d\mu(x)\Big|\\&=\Big|\int_G g(x)\overline{\langle x, \gamma
\rangle}\phi(x)d\mu(x)\Big|=\big|\widehat{(g\phi) d\mu}(\gamma)\big|. \end{align*} Then $\|\widehat{g(\cdot-a)
d\mu}\|_{L^q(\nu)}=\|\widehat{(g\phi) d\mu}\|_{L^q(\nu)}$. We also have \begin{align*} \|Mg-g\phi\|^p_{L^p(\mu)}
&=\int_G|Mg(x)-g(x)\phi(x)|^pd\mu(x)\\&=\int_G|\phi(x)-M|^p|g(x)|^p d\mu(x)\leq\epsilon^p\|g\|^p_{L^p(\mu)}
=\epsilon^p. \end{align*} Then, using the upper (p, q)-frame bound, \begin{align*}\|\widehat{(Mg)d\mu}-\widehat{(g\phi)
d\mu}\|_{L^q(\nu)}=\|\widehat{(Mg-g\phi) d\mu}\|_{L^q(\nu)}\leq\sqrt[q]{B}\|Mg-g\phi\|_{L^p(\mu)}
\leq\sqrt[q]{B}\epsilon. \end{align*} This implies that \begin{align*}\|\widehat{(Mg)d\mu}\|_{L^q(\nu)}-\|\widehat{g(\cdot-a)
d\mu}\|_{L^q(\nu)}&=\|\widehat{(Mg)d\mu}\|_{L^q(\nu)}-\|\widehat{(g\phi)d\mu}\|_{L^q(\nu)}\\&\leq
\|\widehat{(Mg)d\mu}-\widehat{(g\phi) d\mu}\|_{L^q(\nu)}\leq\sqrt[q]{B}\epsilon. \end{align*} Finally, we have
\begin{align*}\frac{\sqrt[q]{A}}{\sqrt[q]{B}}M^{1-\frac{1}{p}}&=\frac{\sqrt[q]{A}M}{\sqrt[q]{B}\sqrt[p]{M}}\leq
\frac{\sqrt[q]{A}\|Mg\|_{L^p(\mu)}}{\sqrt[q]{B}\|g(\cdot-a)\|_{L^p(\mu)}}\leq\frac{\|\widehat{(Mg) d\mu}
\|_{L^q(\nu)}}{\|\widehat{g(\cdot-a) d\mu}\|_{L^q(\nu)}}\\&=1+\frac{ \|\widehat{(Mg) d\mu}
\|_{L^q(\nu)}-\|\widehat{g(\cdot-a) d\mu}\|_{L^q(\nu)}}{\|\widehat{g(\cdot-a) d\mu}\|_{L^q(\nu)}}\\&\leq
1+\frac{\sqrt[q]{B}\epsilon}{\sqrt[q]{A}\|g(\cdot-a)\|_{L^p(\mu)}}\leq 1+\frac{\sqrt[q]{B}\epsilon}{\sqrt[q]{A}
\sqrt[p]{M-\epsilon}}. \end{align*} Letting $\epsilon\to 0$ we get $M^{q(1-\frac{1}{p})}\leq\frac{B}{A}$.\end{proof}
The next proposition shows that the (p, q)-frame measure property is preserved under the restriction of measures on $G$.
We leave the proof to interested readers.
\begin{proposition}\label{pro3.12}
Let $\mu$ be a finite Borel measure on $G$ and $\nu$ be a Borel measure on $\widehat{G}$. Suppose that $E\subset G$
with $\mu(E)>0$. If $\nu$ is a (p, q)-frame measure for $\mu$, then $\nu$ is a (p, q)-frame measure for $\mu|_{E}$ with
the same (p, q)-frame bounds. \end{proposition}
The following lemma provides the connection between two sets of positive Haar measure with their translation in $G$.
We will use it later.
\begin{lemma}\label{lem3.13}
Let $G$ be a $\sigma$-finite LCA group and let $X, Y\subset G$ be two sets of positive Haar measure. Then there
exists a subset $F\subset X$ of positive Haar measure and some $a\in G$ such that $F+a\subset Y$. \end{lemma}\begin{proof}
Taking subsets $X, Y\subset G$ with positive Haar measure, we can assume $X$ and $Y$ are bounded. For each $z\in G$
we have \begin{align*} \chi_X\ast\chi_{-Y}(z)&=\int_G\chi_X(t)\chi_{-Y}(z-t)d\lambda_G(t)\\&=\int_G\chi_X(t)\chi_{Y+z}(t)
d\lambda_G(t)\\&=\lambda_G\big(X\cap(Y+z)\big),\end{align*} where $\lambda_G$ is the Haar measure on $G$.
Apply the Fourier transform on the function $\chi_X\ast\chi_{-Y}$, we have $\widehat{\chi_X\ast\chi_{-Y}}=\widehat{\chi_X}
\cdot\widehat{\chi_{-Y}}\neq 0$. This shows that $\chi_X\ast\chi_{-Y}\neq 0$. Therefore $\chi_X\ast\chi_{-Y}(a)\neq 0$ for
some $a\in G$. So $\lambda_G\big(X\cap(Y+a)\big)>0$. Let $F=X\cap(Y+a)$. Then $F-a\subset Y$ and this proves the the
lemma.\end{proof}
Let $\mu$ be a Borel measure on $G$. Let $f:G\rightarrow\mathbb{R}$ be a non-negative Borel measurable function.
We define the $\mu$-essential supremum of $f$ as follows:\begin{align*} \essup_\mu(f)=\inf\big\{M\in [0, \infty] :\; f\leq M,
\;\mu-\text{a.e.}\big\}. \end{align*} We also define the $\mu$-essential infimum of $f$ as follows: \begin{align*} \esinf_\mu(f)
=\sup\big\{m\in [0, \infty) :\; f\geq m,\;\mu-\text{a.e.}\big\}. \end{align*} \par
The next theorem generalizes a result of Dutkay, Lai \cite{7} to the situation of (p, q)-frame measure.
\begin{theorem}\label{thm3.14}
Let $G$ be a $\sigma$-finite LCA group. Let $\mu$ be a finite Borel measure on $G$ and $\nu$ be a
Borel measure on $\widehat{G}$, respectively. Suppose that $\mu=\phi d\lambda_G$ is an absolutely continuous measure
with respect to $\lambda_G$, where $\lambda_G$ is the Haar measure on $G$. If $\nu$ is a (p, q)-frame measure for
$\mu$ with (p, q)-frame bounds $A, B>0$, then \begin{align*}\frac{B}{A}\geq\Big(\frac{\essup_\mu(\phi)}{\esinf_\mu(\phi)}
\Big)^{q(1-\frac{1}{p})}. \end{align*} In particular, if $\essup_\mu(\phi)=\infty$ or $\esinf_\mu(\phi)=0$, then there is no
(p, q)-frame measure for $\mu$. \end{theorem}\begin{proof}
Let $M=\essup(\phi)$ and $m=\esinf(\phi)$. Suppose in the first part that $m>0$ and $M<\infty$. Fix
$\epsilon>0$. Then there exist sets of positive Haar measure $X, Y\subset G$ such that $m\leq \phi(x)\leq m+\epsilon$
for all $x\in X$, and $M-\epsilon\leq \phi(y)\leq M$ for all $y\in Y$. By the Lemma \ref{lem3.13} there exist a set $F$ of
positive Haar measure and some $a\in G$ such that for all $x\in F$ \begin{align*} m\leq \phi(x)\leq m+\epsilon,
\hspace{0.5cm}\text{and}\hspace{0.5cm}M-\epsilon\leq \phi(x)\leq M. \end{align*}
Now for every Borel subset $E$, we have \begin{align*} T_a(\mu|_{F+a})(E)&=\int_G\chi_E(x-a)d\mu|_{F+a}(x)\\&=
\int_G\chi_F(x-a)\chi_E(x-a)d\mu(x)\\&=\int_G\chi_F(x-a)\chi_E(x-a)\phi(x)d\lambda_G(x)\\&=\int_G\chi_F(x)\chi_E(x)
\phi(x+a)d\lambda_G(x)\\&=\phi(\cdot+a) d\lambda_G|_{F}(E)=\frac{\phi(\cdot+a)}{\phi(\cdot)}d\mu|_{F}(E).\end{align*}
Therefore \begin{align*} \Big\|\frac{dT_a(\mu|_{F+a})}{d\mu}\Big\|_\infty=\Big\|\frac{\phi(x+a)}{\phi(x)}|_{F}\Big\|_\infty
\geq\frac{M-\epsilon}{m+\epsilon}. \end{align*} By letting $\epsilon\to 0$ and using Theorem \ref{thm3.11} we get
\begin{align*} \frac{B}{A}\geq\Big\|\frac{dT_a(\mu|_{F+a})}{d\mu}\Big\|^{q(1-\frac{1}{p})}_\infty\geq
\Big(\frac{M}{m}\Big)^{q(1-\frac{1}{p})}=\Big(\frac{\essup_\mu(\phi)}{\esinf_\mu(\phi)}\Big)^{q(1-\frac{1}{p})}.
\end{align*}
Suppose for the second part that $M=\infty$. Then for any $N$ we can find a subset $F$ of positive Haar measure such that
\begin{align*} N\leq\essup(\phi|_F)<\infty,\hspace{0.5cm}\text{and}\hspace{0.5cm} 0< \esinf(\phi|_F)\leq L, \end{align*}
for some fixed $L$. By the Proposition \ref{pro3.12}, $\nu$ is a (p, q)-frame measure for $\mu|_{F}$ with the same
(p, q)-frame bounds. Then according to the first part argument we have $\frac{B}{A}\geq\big(\frac{N}{L}
\big)^{q(1-\frac{1}{p})}$.  Letting $N\to\infty$ we obtain a contradiction. A similar argument shows that $\esinf(\phi)>0$.
\end{proof}
In \cite{12}, Fu and Lai proved that if $\mu$ and $\lambda$ are two finite Borel measures without atoms and the pair
$(\mu, \lambda)$ is a packing pair of measures, then $\mu\ast\lambda+\delta_g\ast\mu$ does not admit any frame
measure. In the following, we show that these measures also do not admit any (p, q)-frame measure.
First, we collect the main results of the packing pair of measures from \cite{12}.
\begin{theorem}\label{thm3.15}\cite{12}
Let $\mu$ and $\lambda$ be two finite Borel measures without atoms (i.e. all points have measure zero) on $G$ and
$(\mu, \lambda)$ forms a packing pair. Then \begin{enumerate}\item[$(i)$] For any $x, y\in K_\lambda$ and $x\neq y$,
$(K_\mu+x) \cap(K_\lambda+y)=\emptyset$. \item[$(ii)$] For any $E\subset K_\mu$ and $F\subset K_\lambda$,
$\chi_{E+F}(x +y)=\chi_E(x)\chi_F(y)$ for all $x\in K_\mu$ and $y\in K_\lambda$. \item[$(iii)$] Suppose that
$\sigma=\mu\ast\lambda$. Then for any $E\subset K_\mu$ and $F\subset K_\lambda$, \begin{align*}
\widehat{\chi_{E+F}d\sigma}(\gamma)=\big(\widehat{(\chi_Ed\mu)\ast(\chi_Fd\lambda}\big)(\gamma)=
\widehat{\chi_Ed\mu}(\gamma)\widehat{\chi_Fd\lambda}(\gamma).\end{align*} In particular, $\sigma(E+F)
=\mu(E)\lambda(F)$. \item[$(iv)$] Suppose that $\sigma=\mu\ast\lambda$. Then for any $g\in G$,
$\sigma(K_\mu+g)=0$.\end{enumerate}
\end{theorem}
\begin{theorem}\label{thm3.17}
Let $\mu$ and $\lambda$ be two Borel probability measures without atoms on $G$ and let $(\mu, \lambda)$ be a
packing pair and $\sigma=\mu\ast\lambda$. Then for any $g\in G$, the measure $\rho_g=\sigma+\delta_g\ast\mu$
does not admit any (p, q)-frame measure.\end{theorem}
\begin{proof}
We also can assume that $0\in K_\lambda$. In other words, for any open ball $V$ centered around $0$, we have
$\lambda(V)>0$. We now show that $\rho_g$ does not admit any (p, q)-frame measure. Suppose on the
contrary that $\rho_g$ admits a (p, q)-frame measure $\nu$ with frame bounds $A, B$. Suppose that
$\{U_n\}_{n\in\mathbb{N}}$ is an arbitrary countable local base at $0$ with $U_{n+1}\subset U_n$ and
$\lambda (U_1)<\infty$. As $\lambda$ has no atoms, $\lambda(U_n)\to 0$. For $n\in\mathbb{N}$, take $V_n=K_\mu+(U_n\cap K_\lambda)$.
Then by $V_n\subset K_\sigma$ we have $\chi_{V_n}\in L^p(\sigma)$. Furthermore, by Theorem \ref{thm3.15} we have $\sigma(K_\mu+g)=0$.
So, as a consequence of Theorem \ref{thm3.3} and \ref{thm3.10}, $\nu$ is also a (p, q)-frame measure for $\sigma$ and $\mu$ with the same
frame bounds $A, B$. Therefore, we have \begin{align*} A\sigma(V_n)^{\frac{q}{p}}&=A\Big(\int_G
|\chi_{V_n}(x)|^pd\sigma(x)\Big)^{\frac{q}{p}}\leq\int_{\widehat{G}}\big|\widehat{\chi_{V_n}d\sigma}(\gamma)\big|^qd\nu(\gamma)
\\&=\int_{\widehat{G}}\big|\widehat{\chi_{K_\mu}d\mu}(\gamma)
\big|^q\big|\widehat{\chi_{U_n}d\lambda}(\gamma)\big|^qd\nu(\gamma)\\&\leq \big(\lambda(U_n)\big)^q
\int_{\widehat{G}}\big|\widehat{\chi_{K_\mu}d\mu}(\gamma)\big|^qd\nu(\gamma)\\&\leq B\big(\lambda(U_n)\big)^q \Big(\int_G |\chi_{K_\mu}(x)|^pd\mu(x)\Big)^{\frac{q}{p}}=B\big(\lambda(U_n)\big)^q. \end{align*}
Now, by Theorem \ref{thm3.15} we obtain \begin{align*} \sigma(V_n)=\mu(K_\mu)
\lambda(U_n\cap K_\lambda)=\lambda(U_n\cap K_\lambda)=\lambda(U_n). \end{align*} This implies that
\begin{align}\label{eq3.1} \frac{1}{\big(\lambda(U_n)\big)^{q(1-\frac{1}{p})}}\leq\frac{B}{A}. \end{align}
Finally, since $\lambda(U_n)\rightarrow 0$, it follows that the left hand side of (\ref{eq3.1}) can be made
arbitrarily large, while the right hand side remains bounded. This is a contradiction. \end{proof}

\section{Perturbations of (p, q)-frame measures}

The purpose of this section is to study the stability of (p, q)-frame measures. Similar to ordinary frames, (p, q)-frame measures
are stable under small perturbations. The following is an important result of the perturbation of operators that we will use in the
studying of the stability of frame measures.
\begin{proposition}\label{pro4.1}\cite{01}
Let $X, Y$ be two Banach spaces and let $U: X\rightarrow Y$ be a linear operator. Assume that there exist constants
$C, D\in [0; 1[$ such that \begin{align*} \|x-Ux\|\leq C\|x\|+D\|Ux\|,\hspace{0.7cm} \forall x\in X. \end{align*} Then
$U$ is invertible and \begin{align*} \frac{1-C}{1+D}\|x\|\leq\|Ux\|\leq \frac{1+C}{1-D}\|x\|,\hspace{0.3cm}\frac{1-D}{1+C}
\|x\|\leq\|U^{-1}x\|\leq \frac{1+D}{1-C}\|x\|,\hspace{0.3cm}\forall x\in X. \end{align*}
\end{proposition}
We start with the definition of a synthesis operator for a (p, q)-Bessel measure. To show that the inverse Fourier transform
appearing in this formula is well-defined, we need the next proposition.
\begin{proposition}\label{pro4.001}
Let $1<p, q<\infty$ and let $\mu$ be a finite Borel measure on LCA group $G$ and $\nu$ be a Borel measure on
$\widehat{G}$. Suppose that $\nu$ is a (p, q)-Bessel measure for $\mu$. Then for each $\varphi\in L^{q'}(\nu)$, the
mapping $x\mapsto\widecheck{\varphi d\nu}(x)$ is a well-defined from $G$ to $\mathbb{C}$ and $\widecheck{\varphi
d\nu}\in L^{p'}(\mu)$, where $p', q'$ are the conjugate exponents for $p, q$, respectively.
\end{proposition}
\begin{proof} Assume that $\nu$ is a (p, q)-Bessel measure for $\mu$ with Bessel bound $B$. Define the operator \begin{align*}
U:L^p(\mu)\rightarrow L^q(\nu),\hspace{1cm} U(f)=\widehat{fd\mu},\hspace{0.5cm}\forall f\in L^p(\mu).\end{align*}
Then $U$ is a linear bounded operator with $\|U\|\leq\sqrt[q]{B}$, so the adjoint operator \begin{align*} U^*:L^{q'}(\nu)
\rightarrow L^{p'}(\mu), \end{align*} is also a linear bounded operator with $\|U^*\|\leq\sqrt[q]{B}$. In order to find the
expression for $U^*$ let $\varphi\in L^{q'}(\nu)$ and $f\in L^p(\mu)$. Then by the Fubini's theorem we have \begin{align*}
U^*(\varphi)(f)&=\int_{\widehat{G}} \widehat{fd\mu}(\gamma)\overline{\varphi(\gamma)}d\nu(\gamma)\\&=
\int_{\widehat{G}}\int_G\overline{\varphi(\gamma)} f(x) \overline{\langle x, \gamma\rangle}d\mu(x)d\nu(\gamma) \\&=
\int_G f(x)\overline{\int_{\widehat{G}}\varphi(\gamma)\langle x, \gamma\rangle d\nu(\gamma)}d\mu(x)\\&=\int_G f(x)
\overline{\widecheck{\varphi d\nu}(x)}d\mu(x). \end{align*} Thus  \begin{align*} \|\widecheck{\varphi d\nu}\|_{L^{p'}(\mu)}
=\|U^*(\varphi)\|_{L^{p'}(\mu)}\leq\sqrt[q]{B}\|\varphi\|_{L^{q'}(\nu)}.\end{align*}
\end{proof}
\begin{corollary}\label{cor4.0001}
Assume the conditions of Proposition \ref{pro4.001}. Then \begin{align*} \langle\widehat{fd\mu}, \varphi\rangle = \langle f,
\widecheck{\varphi d\nu}\rangle,\hspace{5mm} \forall f\in L^p(\mu),\;\;\; \varphi\in L^{q'}(\nu),\end{align*} where $q'$ is the
conjugate exponent to $q$.
\end{corollary}
\begin{definition}\label{def4.2}
Let $\nu$ be a (p, q)-Bessel measure for $\mu$. The analysis operator for $\mu$ and $\nu$ is defined by\begin{align*}
U_{\mu, \nu}:L^p(\mu)\rightarrow L^q(\nu), \hspace{1cm} U_{\mu, \nu}(f)=\widehat{fd\mu}
\quad\forall f\in L^p(\mu). \end{align*}
\end{definition}
\begin{definition}\label{def4.3}
Let $\nu$ be a (p, q)-Bessel measure for $\mu$. The synthesis operator for $\mu$ and $\nu$ is defined by\begin{align*}
T_{\mu, \nu}:L^{q'}(\nu)\rightarrow L^{p'}(\mu), \hspace{1cm} T_{\mu, \nu}(\varphi)=\widecheck{\varphi d\nu}
\quad\forall \varphi\in L^{q'}(\nu), \end{align*} where $p', q'$ are the conjugate exponents for $p, q$, respectively.
\end{definition}
The following results will provide us with concrete formulas for the (p, q)-Bessel/frame measure operators.
\begin{corollary}\label{cor4.4}
Let $\nu$ be a (p, q)-Bessel measure for $\mu$. Then $U^*_{\mu, \nu}=T_{\mu, \nu}$. \end{corollary}
\begin{corollary}\label{cor4.5}\textsl{•}
Let $\nu$ be a frame measure for $\mu$ with frame bounds $A$ and $B$. Then the operator $T_{\mu, \nu}T^*_{\mu, \nu}$
is invertible and \begin{align*} A\;\textsl{Id}_{L^2(\mu)}\leq T_{\mu, \nu}T^*_{\mu, \nu}\leq B\;\textsl{Id}_{L^2(\mu)}.
\end{align*}\end{corollary}
First, we derive the following result about the stability of frame measures under small perturbations of the associated Borel
measures on $\widehat{G}$.
\begin{theorem}\label{thm4.6}
Let $\mu$ be a finite Borel measure on LCA group $G$ and $\nu, \rho$ be two Borel measures on $\widehat{G}$, respectively.
Let $\nu$ be a frame measure for $\mu$ with frame bounds $A$ and $B$ and $S: L^2(\rho)\rightarrow L^2(\nu)$ be a bounded
invertible operator. Choose $C, D, M\geq 0$ such that $\max\big\{C+\frac{M}{\sqrt{A}\|S^{-1}\|^{-1}}, D\big\}<1$. Further, let
\begin{equation}\label{eq4.1} \|\widecheck{S(\varphi)d\nu}-\widecheck{\varphi d\rho}\|_{L^2(\mu)}\leq C\|\widecheck{S(\varphi)d\nu}\|_{L^2(\mu)}+D\|\widecheck{\varphi d\rho}\|_{L^2(\mu)}+M\|\varphi\|_{L^2(\rho)},\end{equation}
for all $\varphi\in L^2(\rho)$. Then $\rho$ is a frame measure for $\mu$ with frame bounds\begin{align*} A\|S^{-1}\|^{-2}
\Big(1-\frac{C+D+\frac{M}{\sqrt{A}\|S^{-1}\|^{-1}}}{1+D}\Big)^2\hspace{0.2cm} \text{and}\hspace{0.2cm} B\|S\|^2
\Big(1+\frac{C+D+\frac{M}{\sqrt{B}\|S\|}}{1-D}\Big)^2.\end{align*}
\end{theorem}
\begin{proof} We first prove the upper frame bound. For each $\varphi\in L^2(\rho)$, we obtain\begin{align*}
\|\widecheck{\varphi d\rho}\|_{L^2(\mu)}&\leq\|\widecheck{S(\varphi)d\nu}-\widecheck{\varphi d\rho}\|_{L^2(\mu)}+
\|\widecheck{S(\varphi)d\nu}\|_{L^2(\mu)}\\&\leq(1+C)\|\widecheck{S(\varphi)d\nu}\|_{L^2(\mu)}+D\|
\widecheck{\varphi d\rho}\|_{L^2(\mu)}+M\|\varphi\|_{L^2(\rho)}. \end{align*} So\begin{align*}
\|\widecheck{\varphi d\rho}\|_{L^2(\mu)}&\leq\frac{1+C}{1-D}\|\widecheck{S(\varphi)d\nu}\|_{L^2(\mu)}+
\frac{M}{1-D}\|\varphi\|_{L^2(\rho)}. \end{align*}
Since $\nu$ is a frame measure for $\mu$, so the synthesis operator $T_{\mu, \nu}$ is bounded and $\|T_{\mu, \nu}\|
\leq\sqrt{B}$. Let $T_{\mu, \rho}: L^2(\rho)\rightarrow L^2(\mu)$ be the synthesis operator for $\mu$ and $\rho$. Then we
have \begin{align*} \|T_{\mu, \rho}(\varphi)\|_{L^2(\mu)}&\leq \frac{1+C}{1-D}\|T_{\mu, \nu}S(\varphi)\|_{L^2(\mu)}
+\frac{M}{1-D}\|\varphi\|_{L^2(\rho)}\\&\leq\frac{(1+C)\sqrt{B}\|S\|+M}{1-D}\|\varphi\|_{L^2(\rho)}. \end{align*}
Hence \begin{align*} \|T_{\mu, \rho}(\varphi)\|^2_{L^2(\mu)}&\leq B\|S\|^2\Big(1+\frac{C+D+\frac{M}{\sqrt{B}
\|S\|}}{1-D}\Big)^2\|\varphi\|^2_{L^2(\rho)}. \end{align*} This estimate shows that $\rho$ is a Bessel measure for $\mu$
with the upper bound \begin{align*}B\|S\|^2\Big(1+\frac{C+D+\frac{M}{\sqrt{B}\|S\|}}{1-D}\Big)^2.\end{align*}
In order to show the lower frame bound of $\rho$. Since $\nu$ is a frame measure for $\mu$, the operator
$T_{\mu, \nu}T^*_{\mu, \nu}$ is invertible and $\|(T_{\mu, \nu}T^*_{\mu, \nu})^{-1}\|\leq\frac{1}{A}$. If we define
the operator $V:L^2(\mu)\rightarrow L^2(\nu)$ by $V(f)=T^*_{\mu, \nu}(T_{\mu, \nu}T^*_{\mu, \nu})^{-1}(f)$, then
$\|V\|\leq\frac{1}{\sqrt{A}}$. Now take $\varphi=S^{-1}V(f)\in L^2(\rho)$. Then by (\ref{eq4.1}) we have \begin{align*}
\|f-T_{\mu, \rho}S^{-1}V(f)\|_{L^2(\mu)}&=\|T_{\mu, \nu}S(\varphi)-T_{\mu, \rho}(\varphi)\|_{L^2(\mu)}\\&\leq
C\|f\|_{L^2(\mu)}+D\|T_{\mu, \rho}S^{-1}V(f)\|_{L^2(\mu)}+M\|S^{-1}V(f)\|_{L^2(\rho)}\\&\leq
\big(C+\frac{M}{\sqrt{A}\|S^{-1}\|^{-1}}\big)\|f\|_{L^2(\mu)}+D\|T_{\mu, \rho}S^{-1}V(f)\|_{L^2(\mu)}. \end{align*}
By Proposition \ref{pro4.1}, the operator $T_{\mu, \rho}S^{-1}V$ is invertible and \begin{align*} \|(T_{\mu, \rho}S^{-1}V)^{-1}\|
\leq\frac{1+D}{1-\big(C+\frac{M}{\sqrt{A}\|S^{-1}\|^{-1}}\big)}.\end{align*} Finally, for each $f\in L^2(\mu)$ we have
\begin{align*} \|f\|^4_{L^2(\mu)}&=|\langle f, f\rangle_{L^2(\mu)}|^2=|\langle T_{\mu, \rho}S^{-1}V(T_{\mu, \rho}
S^{-1}V)^{-1}(f), f\rangle_{L^2(\mu)}|^2\\&=|\langle S^{-1}V(T_{\mu, \rho}S^{-1}V)^{-1}(f), \widehat{fd\mu}\rangle_{L^2(\rho)}|^2\\&\leq\frac{\|S^{-1}\|^2}{A}\|(T_{\mu, \rho}S^{-1}V)^{-1}(f)\|^2_{L^2(\rho)}
\|\widehat{fd\mu}\|^2_{L^2(\rho)}\\&\leq\frac{\|S^{-1}\|^2}{A}\Big(\frac{1+D}{1-\big(C+\frac{M}{\sqrt{A}
\|S^{-1}\|^{-1}}\big)}\Big)^2\|f\|^2_{L^2(\mu)}\|\widehat{fd\mu}\|^2_{L^2(\rho)}.\end{align*}
Dividing both sides of this inequality by $\frac{\|S^{-1}\|^2}{A}\Big(\frac{1+D}{1-\big(C+\frac{M}{\sqrt{A}
\|S^{-1}\|^{-1}}\big)}\Big)^2$ completes the proof.
\end{proof}
The following theorem is about the stability of measures under small perturbations which have a (p, q)-frame measure.
\begin{theorem}\label{thm4.7}
Let $\mu, \lambda$ be two finite Borel measures on LCA group $G$ and $\nu$ be a Borel measure on $\widehat{G}$, respectively.
Let $\nu$ be a (p, q)-frame measure for $\mu$ with frame bounds $A, B$, and $S: L^p(\lambda)\rightarrow L^p(\mu)$
be a bounded invertible operator. Choose $C, D, M\geq 0$ such that $\max\big\{C+\frac{M}{\sqrt{A}\|S^{-1}\|^{-1}}, D\big\}<1$.
 Further, let \begin{equation}\label{eq4.2} \|\widehat{S(f)d\mu}-\widehat{fd\lambda}\|_{L^q(\nu)}\leq C\|\widehat{S(f)d\mu}
 \|_{L^q(\nu)}+D\|\widehat{fd\lambda}\|_{L^q(\nu)}+M\|f\|_{L^p(\lambda)},\end{equation}
for all $f\in L^p(\lambda)$. Then $\nu$ is a (p, q)-frame measure for $\lambda$ with frame bounds\begin{align*}
A\|S^{-1}\|^{-q}\Big(1-\frac{C+D+\frac{M}{\sqrt[q]{A}\|S^{-1}\|^{-1}}}{1+D}\Big)^q\hspace{0.2cm} \text{and}
\hspace{0.2cm} B\|S\|^q \Big(1+\frac{C+D+\frac{M}{\sqrt[q]{B}\|S\|}}{1-D}\Big)^q.\end{align*}
\end{theorem}
\begin{proof} We first prove the upper frame bound. For each $f\in L^p(\lambda)$, we obtain
\begin{align*}\|\widehat{fd\lambda}\|_{L^q(\nu)}&\leq\|\widehat{S(f)d\mu}-\widehat{fd\lambda}\|_{L^q(\nu)}+
\|\widehat{S(f)d\mu}\|_{L^q(\nu)}\\&\leq(1+C)\|\widehat{S(f)d\mu}\|_{L^q(\nu)}+D\|\widehat{fd\lambda}\|_{L^q(\nu)}
+M\|f\|_{L^p(\lambda)}. \end{align*} Hence \begin{align*} \|\widehat{fd\lambda}\|_{L^q(\nu)}&\leq
\frac{1+C}{1-D}\|\widehat{S(f)d\mu}\|_{L^q(\nu)}+\frac{M}{1-D}\|f\|_{L^p(\lambda)}\\&\leq
\frac{1+C}{1-D}\sqrt[q]{B}\|S\|\|f\|_{L^p(\lambda)}+\frac{M}{1-D}\|f\|_{L^p(\lambda)}\\&=
\frac{(1+C)\sqrt[q]{B}\|S\|+M}{1-D}\|f\|_{L^p(\lambda)}. \end{align*} So  \begin{align*} \|\widehat{fd\lambda}
\|^q_{L^q(\nu)}&\leq B\|S\|^q\Big(1+\frac{C+D+\frac{M}{\sqrt[q]{B}\|S\|}}{1-D}\Big)^q\|f\|^q_{L^p(\lambda)}. \end{align*}
To prove the lower frame bound, we compute
\begin{align*}
\|\widehat{fd\lambda}\|_{L^q(\nu)}&\geq\|\widehat{S(f)d\mu}\|_{L^q(\nu)}-\|\widehat{S(f)d\mu}-\widehat{fd\lambda}
\|_{L^q(\nu)}\\&\geq(1-C)\|\widehat{S(f)d\mu}\|_{L^q(\nu)}-D\|\widehat{fd\lambda}\|_{L^q(\nu)}-M\|f\|_{L^p(\lambda)}.
\end{align*} Hence
\begin{align*}(1+D)\|\widehat{fd\lambda}\|_{L^q(\nu)}&\geq (1-C)\|\widehat{S(f)d\mu}\|_{L^q(\nu)}-M\|f\|_{L^p(\lambda)}
\\&\geq (1-C)\sqrt[q]{A}\|S(f)\|_{L^p(\mu)}-M\|f\|_{L^p(\lambda)}\\&\geq\big((1-C)\sqrt[q]{A}\|S^{-1}\|^{-1}-M\big)
\|f\|_{L^p(\lambda)}. \end{align*} Therefore
\begin{align*}\|\widehat{fd\lambda}\|^q_{L^q(\nu)}&\geq\Big(\frac{(1-C)\sqrt[q]{A}\|S^{-1}\|^{-1}-M}{1+D}\Big)^q
\|f\|^q_{L^p(\lambda)}\\&\geq A\|S^{-1}\|^{-q}\Big(\frac{1-C-\frac{M}{\sqrt[q]{A}\|S^{-1}\|^{-1}}}{1+D}\Big)^q
\|f\|^q_{L^p(\lambda)}\\&\geq A\|S^{-1}\|^{-q}\Big(1-\frac{C+D+\frac{M}{\sqrt[q]{A}\|S^{-1}\|^{-1}}}{1+D}\Big)^q
\|f\|^q_{L^p(\lambda)}. \end{align*}\end{proof}
As a consequence of Theorem \ref{thm4.7}, we have the following result.
\begin{corollary}\label{cor4.8}
Let $\mu, \lambda$ be two finite Borel measures on $G$ and let $\nu$ be a (p, q)-frame measure for $\mu$ with frame bounds
$A, B$. Suppose that $S: L^p(\lambda)\rightarrow L^p(\mu)$ is a bounded invertible operator. Choose $C, D, M\geq 0$ such
that $\max\{C, D, M\}<1$. Further, let
\begin{align*} \|\widehat{S(f)d\mu}-\widehat{fd\lambda}\|^2_{L^q(\nu)}\leq C\|\widehat{S(f)d\mu}\|^2_{L^q(\nu)}
&+2D\|\widehat{S(f)d\mu}\|_{L^q(\nu)}\|\widehat{fd\lambda}\|_{L^q(\nu)}\\&+M\|\widehat{fd\lambda}\|^2_{L^q(\nu)},
\end{align*} for all $f\in L^p(\lambda)$. Then $\nu$ is a (p, q)-frame measure for $\lambda$ with frame bounds
\begin{align*} A\|S^{-1}\|^{-q}\Big(\frac{1-\sqrt{\eta}}{1+\sqrt{\eta}}\Big)^q
\hspace{0.2cm} \text{and}\hspace{0.2cm} B\|S\|^q\Big(\frac{1+\sqrt{\eta}}{1-\sqrt{\eta}}\Big)^q.
\end{align*}\end{corollary}
\begin{proof} This claim follows immediately from the fact that for each $f\in L^p(\lambda)$ and $\eta=\max\{C, D, M\}$ we have
\begin{align*} \|\widehat{S(f)d\mu}&-\widehat{fd\lambda}\|_{L^q(\nu)}\leq \sqrt{\eta}\big(\|\widehat{S(f)d\mu}\|_{L^q(\nu)}
+\|\widehat{fd\lambda}\|_{L^q(\nu)}\big).\end{align*} Now the conclusion follows from Theorem \ref{thm4.7}.\end{proof}
\begin{theorem}\label{thm4.9}
Let $\mu, \lambda$ be two finite Borel measures on LCA group $G$ and let $\nu$ be a (p, q)-frame measure for $\mu$.
Suppose that $S: L^p(\lambda)\rightarrow L^p(\mu)$ is a bounded invertible operator.
Then $\nu$ is a (p, q)-frame measure for $\lambda$ if and only if there exists a constant $M>1$ such that \begin{equation}
\label{eq4.3} \|\widehat{S(f)d\mu}-\widehat{fd\lambda}\|_{L^q(\nu)}\leq M\min\big\{\|\widehat{S(f)d\mu}\|_{L^q(\nu)},
\|\widehat{fd\lambda}\|_{L^q(\nu)}\big\},\hspace{0.5cm} \forall f\in L^p(\lambda). \end{equation}\end{theorem}
\begin{proof} Let $\nu$ be a (p, q)-frame measure for $\mu$ and $\lambda$ with frame bounds $A, B$ and $C, D$, respectively.
Then by the (p, q)-frame measure inequalities we have \begin{align*} \|\widehat{S(f)d\mu}-\widehat{fd\lambda}\|_{L^q(\nu)}
\leq\Big(1+\frac{\sqrt[q]{D}\|S^{-1}\|}{\sqrt[q]{A}}\Big)\|\widehat{S(f)d\mu}\|_{L^q(\nu)}.\end{align*} Similarly, we obtain
\begin{align*} \|\widehat{S(f)d\mu}-\widehat{fd\lambda}\|_{L^q(\nu)}\leq\Big(1+\frac{\sqrt[q]{B}\|S\|}{\sqrt[q]{C}}
\Big)\|\widehat{fd\lambda}\|_{L^q(\nu)}.\end{align*} Choosing $M=1+\max\Big\{\frac{\sqrt[q]{D}\|S^{-1}\|}{\sqrt[q]{A}},
\frac{\sqrt[q]{B}\|S\|}{\sqrt[q]{C}}\Big\}$, completes the proof. For the converse, suppose that $A, B$ are the frame bounds for
(p, q)-frame measure $\nu$ for $\mu$. Then for each $f\in L^p(\lambda)$ we have \begin{align*} \sqrt[q]{A}\|S^{-1}\|^{-1}
\|f\|_{L^p(\lambda)} &\leq\sqrt[q]{A}\|S(f)\|_{L^p(\mu)}\leq\|\widehat{S(f)d\mu}\|_{L^q(\nu)}\\&\leq
\|\widehat{S(f)d\mu}-\widehat{fd\lambda}\|_{L^q(\nu)}+\|\widehat{fd\lambda}\|_{L^q(\nu)}\\&\leq(1+M)
\|\widehat{fd\lambda}\|_{L^q(\nu)}\\&\leq(1+M)\Big(\|\widehat{S(f)d\mu}-\widehat{fd\lambda}\|_{L^q(\nu)}+
\|\widehat{S(f)d\mu}\|_{L^q(\nu)}\Big)\\&\leq(1+M)^2\|\widehat{S(f)d\mu}\|_{L^q(\nu)}\leq\sqrt[q]{B}(1+M)^2
\|S(f)\|_{L^p(\mu)}\\&\leq\sqrt[q]{B}(1+M)^2\|S\|\|f\|_{L^p(\lambda)}.\end{align*} Thus \begin{align*}\frac{\sqrt[q]{A}
\|S^{-1}\|^{-1}}{1+M}\|f\|_{L^p(\lambda)}\leq\|\widehat{fd\lambda}\|_{L^q(\nu)}\leq\sqrt[q]{B}(1+M)\|S\|
\|f\|_{L^p(\lambda)}. \end{align*}\end{proof}


\begin{thebibliography}{10}

\bibitem {CAB} C. Cabrelli, V. Paternostro, \textit{Shift-invariant spaces on LCA groups}, J. Funct. Anal., 258 (2010) 2034--2059.
\bibitem {CAZ} P. G. Cazassa, O. Christensen, \textit{Perturbation of operators and applications to frame theory}, J. Fourier Anal. Appl., 3 (5) (1997) 543--557.
\bibitem {CHR} O. Christensen, \textit{An introductoin to frames and Riesz bases},  Birkh\"{a}user, Switzerland 2016.
\bibitem {CHRS} O. Christensen, S. S. Goh, \textit{Fourier-like frames on locally compact abelian groups}, J. Approx. Theory,
192 (2015) 82--101.
\bibitem {CHRI} O. Christensen, C. Heil, \textit{Perturbations of Banach frames and atomic decompositions}, Math. Nach., 185 (1997) 33--47.
\bibitem {DEV} R. A. Devore, B. Jawerth, B. J. Lucier, \textit{Image compression through wavelet transform coding}, IEEE Trans. on Inform. Theory, 38 (1992) 719--746.
\bibitem {DUF} R. Duffin, A. Schaeffer, \textit{A class of non-harmonic Fourier series}, Trans. Amer. Math. Soc. 72 (1952) 341--366.
\bibitem {DUT} D. E. Dutkay, D. Han, E. Weber, \textit{Continuous and discrete Fourier frames for fractal measures}, Trans. Amer.
Math. Soc. 366 (3) (2014) 1213--1235.
\bibitem {DUTK} D. E. Dutkay, C. K. Lai, \textit{Uniformity of measures with Fourier frames}, Adv. Math. 252 (2014) 684--707.
\bibitem {FEI} H. G. Feichtinger, G. Zimmermann, \textit{A Banach space of test functions for Gabor analysis}.
In H. G. Feichtinger and T. Strohmer, editors, Gabor Analysis and Algorithms: Theory and Applications. Birkh\"{a}user, Boston, (1998) 123--170.
\bibitem {FOL} G. B. Folland: \textit{A course in abstract harmonic analysis}, CRC Press, New York, 2015.
\bibitem {FUG} B. Fuglede, \textit{Commuting self-adjoint partial differential operators and a group theoretic problem}, J. Funct. Anal., 16 (1974), 101--121.
\bibitem {FU} X. Fu, C. K. Lai, \textit{Translational absolute continuity and Fourier frames on a sum of singular measures},
J. Funct. Anal. 274 (2018) 2477--2498.
\bibitem {GAB} J.-P. Gabardo, D. Han, \textit{Frames associated with measurable spaces}, Adv. Comput. Math. 18 (2003) 127--147.
\bibitem {GRO} K. Gr\"{o}chenig, \textit{Aspects of Gabor analysis on locally compact abelian groups}, in: Gabor analysis and algorithms, Birkh\"{a}user, Boston, MA, (1998) 211--231.
\bibitem {GROC} K. Gr\"{o}chenig, \textit{Describing functions: atomic decompositions versus frames}, Monatsh. Math., 112 (1991) 1--42.
\bibitem {HE} X. G. He, C. K. Lai, K. S. Lau, \textit{Exponential spectra in $L^2(\mu)$}, Appl. Comput. Harmon. Anal. 34 (3) (2013) 327--338.
\bibitem {JOR} P. Jorgensen, S. Pedersen, \textit{Dense analytic subspaces in fractal $L^2$ spaces.}, J. Anal. Math., 75 (1998), 185--228.
\bibitem {LAB} I. Laba, Y. Wang, \textit{Some properties of spectral measures}, Appl. Comput. Harmon. Anal., 20 (1) (2006) 149--157.
\bibitem {LAI} C. K. Lai, \textit{On Fourier frame of absolutely continuous measures}, J. Funct. Anal., 261 (10) (2011) 2877--2889.

\bibitem {LEV} N. Lev, \textit{Fourier frames for singular measures and pure type phenomena}, Proc. Amer. Math. Soc., 146 (7) (2018)
2883--2896
\bibitem {NIT} S. Nitzan, A. Olevskii, A. Ulanovskii, \textit{Exponential frames on unbounded sets}, Proc. Amer. Math. Soc. 144 (1) (2016) 109--118.
\bibitem {MAL} S. Mallat, \textit{A Wavelet Tour of Signal Processing}, Academic Press, New York 1998.
\bibitem {MAR} A. W. Marcus, D. A. Spielman, N. Srivastava, \textit{Interlacing families II: Mixed characteristic polynomials and the
Kadison-Singer problem}, Ann. of Math. (2) 182 (1) (2015) 327--350.
\bibitem {RUD} W. Rudin, \textit{Fourier Analysis on Groups}, John Wiley, New York 1990.
\bibitem {STR} R. S. Strichartz, \textit{Mock Fourier series and transforms associated with certain Cantor measures}, J. Anal. Math., 81 (2000)
209--238.
\bibitem {YOU} R. M. Young, \textit{An introduction to nonharmonic Fourier series}, 1st ed., Academic Press, Inc., San Diego, CA, 2001.

\end{thebibliography}
\end{document}